\documentclass{amsart}
\usepackage[margin=1in]{geometry}
\usepackage{amsmath}
\usepackage{amssymb}
\usepackage{amsthm}
\usepackage{amscd}
\usepackage{enumerate}

\usepackage{graphicx}
\usepackage{amsmath,amsthm,amsfonts,amscd,amssymb,comment,eucal,latexsym,mathrsfs}
\usepackage{stmaryrd}
\usepackage[all]{xy}

\usepackage{epsfig}

\usepackage[all]{xy}
\xyoption{poly}
\usepackage{fancyhdr}
\usepackage{wrapfig}
\usepackage{epsfig}

\theoremstyle{plain}
\newtheorem{thm}{Theorem}[section]
\newtheorem{prop}[thm]{Proposition}
\newtheorem{lem}[thm]{Lemma}
\newtheorem{cor}[thm]{Corollary}

\theoremstyle{definition}
\newtheorem{defn}[thm]{Definition}
\theoremstyle{remark}

  \def\C{{\mathbb{C}}}           \def\N{{\mathbb{N}}}    \def\R{{\mathbb{R}}}  \def\T{{\mathbb{T}}}      \def\Z{{\mathbb{Z}}}

\newcommand\Hom{\operatorname{Hom}}

\newcommand\id{\operatorname{id}}

\newcommand\Map{{\operatorname{Map}}}

\newcommand\Prob{\operatorname{Prob}}

\newcommand\topo{\operatorname{top}}

\def\cc{{\curvearrowright}}
\newcommand{\actson}{\curvearrowright}

\newcommand{\acston}{\actson}

\newcommand{\ip}[1]{\langle #1 \rangle}

\begin{document}

\title[Local and Doubly Empirical Convergence and the Entropy of Algebraic Actions]{Local and Doubly Empirical Convergence and the Entropy of Algebraic Actions of Sofic Groups}      
\author{Ben Hayes}
\address{Stevenson Center\\
         Nashville, TN 37240}
\email{benjamin.r.hayes@vanderbilt.edu}
\date{\today}
\maketitle
\begin{abstract} Let $G$ be a sofic group and $X$ a compact group with $G\actson X$ by automorphisms. Using (and reformulating) the notion of local and doubly empirical convergence developed by Austin, we show that in many cases the topological and the measure-theoretic entropy with respect to the Haar measure of $G\actson X$ agree. Our method of proof recovers all known examples. Moreover, the proofs are direct and do not go through explicitly computing the measure-theoretic or topological entropy.

\end{abstract}
\tableofcontents

\section{Introduction}
This paper is concerned with studying entropy for actions of sofic groups on compact groups by automorphisms. Entropy for probability measure-preserving actions of sofic groups was initiated in ground-breaking work of Bowen in \cite{Bow} under the assumption of a finite generating partition. Work of Kerr-Li in \cite{KLi} removed this assumption, defined topological entropy for actions of sofic groups, and proved a variational principle for entropy. The work of Bowen, Kerr-Li thus extended the classical notion of entropy for actions of amenable groups to the vastly larger class of sofic groups.

One class of actions whose entropy theory has received much interest is the case of automorphisms on compact groups. Let $G$ be a countable, discrete, sofic group, let $X$ be a compact group and suppose that $G\actson X$ by automorphisms. One can think of this as a \emph{topological} dynamical system and also a probability measure-preserving system, giving $X$ the Haar measure (which we denote by $m_{X}$). When $X$ is abelian, such an action is called an \emph{algebraic action}. When $G$ is amenable, Deninger proved (see \cite{Den} Theorem 3) that the topological and measure-theoretic entropy coincide for an action of $G$ on a compact group by automorphisms.

While such a general equality of entropies is impossible in the nonamenable case (e.g. if $X$ is finite and $G$ is the free group on two generators there always sofic approximations for which the entropy of $G\actson (X,m_{X})$ is $-\infty$), certain cases of such an equality have been proved for general sofic groups. Let $f\in M_{m,n}(\Z(G))$ and define  $r(f)\colon \Z(G)^{\oplus m}\to \Z(G)^{\oplus n}$ by
\[(rf)\xi)(j)=\sum_{l=1}^{m} \xi(l)f_{lj},\mbox{ for $1\leq j\leq n$}.\]
Let $X_{f}$ be the Pontryagin dual of $\Z(G)^{\oplus n}/(r(f)(\Z(G)^{\oplus m}))$ (i.e. the space of continuous homomorphisms $\Z(G)^{\oplus n}/(r(f)(\Z(G)^{\oplus m}))$ into $\T=\R/\Z$), and let $G\actson X_{f}$ by
\begin{equation}\label{E:algactionintro}
(g\chi)(a)=\chi(g^{-1}a),\mbox{ $a\in \Z(G)^{\oplus n}/(r(f)(\Z(G)^{\oplus m}))$}.
\end{equation}
When $m=n=1$ this is called a \emph{principal algebraic action}. Work of Bowen-Li \cite{BowenLi}, Kerr-Li, and Bowen (\cite{KLi},\cite{BowenEntropy}) proved the equality of topological and measure-theoretic entropy for  principal algebraic actions, when $G$ is residually finite and under some restrictions $f.$ Specifically \cite{KLi},\cite{BowenEntropy} require that $f$ has a convolution inverse in $\ell^{1}$ (though the computation of topological entropy in \cite{KLi} only requires that $f$ is invertible in the full $C^{*}$-algebra of $G$). The work \cite{BowenLi} require that $f$ is (up to multiplicative constant) a Laplace operator (which, if $G$ is nonamenable, implies that $f$ is invertible in the group von Neumann algebra of $G$). We mention our previous work in \cite{Me5}, for which we need more notation. For a countable, discrete group $G$ and $f\in \C(G),$ define $\lambda(f)\colon \ell^{2}(G)\to \ell^{2}(G)$ by
\[(\lambda(f)\xi)(g)=\sum_{h\in G}f(h)\xi(h^{-1}g),\mbox{ for all $g\in G$.}\]
For $A\in M_{m,n}(\C(G)),$ we define $\lambda(A)\colon \ell^{2}(G)^{\oplus n}\to \ell^{2}(G)^{\oplus m}$ by
\[(\lambda(A)\xi)(l)=\sum_{j=1}^{n} \lambda(A_{lj})\xi(j),\mbox{ for $1\leq l\leq m$}.\]
Our previous work in \cite{Me5} proved equality of measure-theoretic and topological entropy for  $G\actson X_{f}$ under the assumption that $f\in M_{m,n}(\Z(G)),$ and that $\lambda(f)$ has dense image, for an arbitrary sofic group $G$ (additionally it is not required that $m=n=1,$ though it is implied that $m\leq n$). Each of \cite{KLi},\cite{BowenLi},\cite{Me5} proved this result indirectly, by explicitly computing both the topological and measure-theoretic entropy to be equal to the Fuglede-Kadison determinant. In \cite{KLi}, Problem 7.7 Kerr-Li asked if there is a direct way to see this equality without having to compute both entropies. Progress on this problem was made by Gaboriau-Seward in \cite{GabSew}, who gave the first general results on equality of topological and measure-theoretic entropy without directly computing either entropy. Specifically, they showed equality of the entropies when $G\actson X$ by automorphisms, where $X$ is any profinite group and when the homoclinic group of $G\actson X$ is dense.

In this paper, we present a proof of this equality under fairly general hypotheses. Our work uses the notion of local and
doubly empirical convergence defined  in \cite{AustinAdd}. In \cite{AustinAdd}, this  is called doubly quenched convergence. Private communication with Austin (as well as other experts in the field) has led us to the opinion that locally and doubly empirical convergence (a term suggested by Austin) is better used to describe this property. The reader is invited to see the introduction and Section 2 of \cite{AustinEP} for reasons behind the name change. In this paper, we shall mostly use the term ``local and doubly empirical convergence'' (resp. `local and empirical convergence'') instead of ``doubly quenched'' (resp. ``quenched''). The main exceptions to this rule are when we will directly reference definitions, theorems, etc. in \cite{AustinAdd}, and in this case we will use the same terminology as in that paper.

We briefly describe local and doubly empirical convergence. By definition, soficity of $G$ allows one to find a sequence of integers $d_{i},$ and maps $\sigma_{i}\colon G\to S_{d_{i}}$ (here $S_{n}$ is the group  of permutations of $\{1,\dots,n\}$) which give  good finitary models of the action of $G$ on itself. These maps are not assumed to be homomorphisms, but this defect of not being a homomorphisms only occurs on a small portion of $\{1,\dots,d_{i}\}.$  Moreover, if one pretends that these maps give actions of $G,$ then for any $g\in G$ the set of points in $\{1,\dots,d_{i}\}$ which are fixed by $\sigma_{i}(g)$  again only occupies a small portion of $\{1,\dots,d_{i}\}$. Briefly, we can then think of these maps as giving ``almost free almost actions." 

	Now fix a compact space $X$ and let $G\actson X$ by homeomorphisms fixing a probability measure $\mu.$ The sofic entropy of $G\actson (X,\mu)$ counts the exponential growth rate of ``how many" maps $\phi\colon \{1,\dots,d_{i}\}\to X$ there are which are almost equivariant for the almost action of $G$ on $\{1,\dots,d_{i}\}$ and approximately pushforward the uniform measure to $\mu$. We will call these maps ``microstates" (this is a heuristic term that will not be defined precisely).
	Typically, these microstates for $G\actson (X,\mu)$ are produced by a probabilistic argument. In many cases this probabilistic argument is no accident: if $G$ is residually finite, if $G\actson (X,\mu)$ is ergodic and has a finite generating partition,  Bowen in \cite{BowenEntropy} shows how one can compute sofic entropy in terms of certain probability measures on $X^{d_{i}}$ which in some sense model  $\mu$. Moreover, the proof in \cite{BowenEntropy} easily extends to the case that $G$ is merely sofic instead of residually finite.  Austin's work in \cite{AustinAdd} modifies the methods in \cite{BowenEntropy}, considering a slightly different way to say that a sequence of measures $\mu_{i}\in \Prob(X^{d_{i}})$ models $\mu.$ 
	
		Roughly, the main difference is that Bowen requires that 
\[\lim_{i\to\infty} \frac{1}{d_{i}}\sum_{j=1}^{d_{i}}\int_{X^{d_{i}}}f(x(j))\,d\mu_{i}(x)=\int f\,d\mu\]
for all $f\in C(X),$  and that Austin requires
\[\int_{X^{d_{i}}}f(x(j))\,d\mu_{i}(x)\approx \int f\,d\mu\mbox{ for ``most'' $j\in\{1,\dots,d_{i}\}$},\]
and  every $f\in C(X).$ Thus Bowen's approach is to require convergence of $\int_{X^{d_{i}}}f(x(j))\,d\mu_{i}(x)$ to $\int f\,d\mu$ on average and Austin's requirement is that $\int_{X^{d_{i}}}f(x(j))\,d\mu_{i}(x)$ is close to $\int f\,d\mu$  with high probability in $j$. This is what Austin in \cite{AustinAdd} calls the \emph{local weak$^{*}$ convergence} of $\mu_{i}$ to $\mu,$ which is denoted $\mu_{i}\to^{lw^{*}}\mu.$ Strictly speaking, this is not the definition Austin gives in \cite{AustinAdd}. Austin works with invariant measures on the left shift space $X^{G}$ for some compact, metrizable space $X$. He also defines a notion of ``local and empirical convergence'' (called ``quenched convergence'' in \cite{AustinAdd}) convergence, which demands local weak$^{*}$ convergence of $\mu_{i}$  to $\mu$ and that $\mu_{i}$ be mostly concentrated on the space of microstates in Bowen's sense (see Section \ref{S:reform} for the precise definition). We shall spend much of Section \ref{S:reform}  translating between the version of local and empirical convergence describe above and the version given in \cite{AustinAdd}. Also considered in \cite{AustinAdd} is the notion of local and doubly empirical convergence (there called ``doubly quenched convergence''), which requires that $\mu_{i}\otimes \mu_{i}$ locally and  empirically converges to $\mu\otimes \mu.$
The work in \cite{AustinAdd} connects local and doubly empirical convergence  in an important way to various additivity properties of entropy and also gives a new approach to sofic entropy, based directly on the consideration of measures $\mu_{i}\in\Prob(X^{d_{i}})$ which locally and doubly empirically converge to $\mu.$  Local and doubly empirical convergence of measures ends up being the right tool to prove the equality of measure-theoretic and topological entropy.

	To state our results, we need the notion of entropy in the presence. Suppose that $G\actson (X,\mu)$ is a probability measure-preserving action and that $G\actson (Y,\nu)$ is a factor of $G\actson (X,\mu)$ with factor map $\pi.$ Entropy in the presence, denoted $h_{(\sigma_{i})_{i},\mu}(Y:X,G),$  measures ``how many" microstates for $G\actson (Y,\nu)$ there are which lift to microstates for $G\actson (X,\mu).$ Entropy in the presence stands in contrast to the usual entropy which just measures ``how many'' microstates for $G\actson (Y,\nu)$ there are (with no regard as to whether or not the come from microstates for $G\actson (X,\mu)$ or not). Note that $h_{(\sigma_{i})_{i},\mu}(X:X,G)$ is just the entropy of $G\actson (X,\mu),$ which we denote by $h_{(\sigma_{i})_{i},\mu}(X,G).$  Entropy in the presence was implicitly defined by Kerr in \cite{KerrPartition} and our previous work in \cite{Me9} showed how one can compute entropy in the presence in terms of a given topological model. 
	
	Suppose that we have a compact space $X,$ an action $G\actson X$ by homeomorphisms,  and  a topological factor $G\actson Y.$ We can also consider the topological entropy in the presence, which we denote by $h_{(\sigma_{i})_{i},\topo}(Y:X,G).$ Topological entropy in the presence is defined in a very analogous way to measure-theoretic entropy in the presence. Again in the case that $Y=X$ this is just the topological entropy, which we denote by $h_{(\sigma_{i})_{i},\topo}(X,G).$ Topological entropy in the presence was defined in  Definition 9.3 of \cite{LiLiang2}. and the formulation of topological entropy in the presence is close to what we have developed in \cite{Me9} (see also Definitions \ref{D:topentrpresence} and \ref{D:measentpresence} in this paper). In \cite{LiLiang2}, Li-Liang use the term ``entropy relative to an extension," but we will use the term ``entropy in the presence" to be consistent with our earlier terminology. Additionally, in later work we will develop relative entropy for actions of sofic groups and we do not want to confuse ``entropy in the presence" with relative entropy. As we showed in \cite{Me9}, if $G$ is amenable, then $h_{(\sigma_{i})_{i},\mu}(Y:X,G)=h_{\mu}(Y,G)$ where $h_{\mu}(Y,G)$ is the classical entropy defined for amenable groups. We now state the main theorem of the paper.

\begin{thm}\label{T:mainintro} Let $G$ be a countable, discrete, sofic group with sofic approximation $\sigma_{i}\colon G\to S_{d_{i}}.$ Let $X$ be a compact, metrizable group with $G\actson X$ by automorphisms. Suppose that there exists $\mu_{i}\in \Prob(X^{d_{i}})$ with $\mu_{i}\to^{lde}m_{X}.$ Then for any closed, normal, $G$-invariant subgroup $Y\subseteq X$ we have
\[h_{(\sigma_{i}),\topo}(X/Y:X,G)=h_{(\sigma_{i})_{i},m_{X}}(X/Y:X,G).\]
In particular, (taking $Y=\{1\}$) the topological entropy (with respect to $(\sigma_{i})_{i}$) of $G\cc X$ and the measure-theoretic entropy (with respect to $(\sigma_{i})_{i}$) of  $G\cc (X,m_{X})$ agree.
\end{thm}

It is shown in  Corollary 5.18 of \cite{AustinAdd} that for a probability measure-preserving action $G\actson (X,\mu)$ the existence of a sequence of measures which locally and doubly empirically convergence converge to $\mu$ is preserved under measure-theoretic factors. It follows that in Theorem \ref{T:mainintro} we in fact have
\[h_{(\sigma_{i})_{i},\topo}(X/Y,G)=h_{(\sigma_{i})_{i},m_{X}}(X/Y,G)\]
for every closed, normal, $G$-invariant subgroup $Y\subseteq X.$ The ability to take factors by normal subgroups in Theorem \ref{T:mainintro} will be important in our follow-up paper, where we will use Theorem \ref{T:mainintro} and our previous results in \cite{Me7} to give new examples of algebraic actions of nonamenable groups which have completely positive entropy in the presence. This will require an investigation of the structural properties of the Outer Pinsker factor for sofic groups (whose definition is folklore, but was first written down in \cite{Me9}).

We remark that  the existence of a sequence $\mu_{i}\in \Prob(X^{d_{i}})$ which locally and doubly empirically converge to $\mu$ is automatic in the amenable case. Thus our results recover the result of Deninger in \cite{Den} on equality of topological and measure-theoretic entropy for actions on compact groups. Moreover, the existence of a sequence $\mu_{i}\in\Prob(X^{d_{i}})$ which locally and doubly empirically converge to $\mu$ is implicitly shown in every case where entropy has been computed. Specifically, we have the following proposition.

 \begin{prop}\label{P:mainexamples} Let $G$ be a countable, discrete, sofic group with sofic approximation $\sigma_{i}\colon G\to S_{d_{i}}.$ For each of the following probability measure-preserving actions $G\actson (X,\mu),$ there exists a sequence $\mu_{i}\in \Prob(X^{d_{i}})$ with $\mu_{i}\to^{lde}\mu$:
 \begin{enumerate}[(i)]
 \item $G\actson (X,\mu)$ any Bernoulli action (implicitly shown in \cite{Bow} and explicitly shown in \cite{AustinAdd}),
 \item $X=X_{f},\mu=m_{X_{f}}$ for some $f\in M_{m,n}(\Z(G)),$ with action given by (\ref{E:algactionintro}) and with $\lambda(f)$ having dense image (implicit in \cite{Me5}),\label{I:finpresintro}
 \item $X=X_{f},\mu=m_{X_{f}}$ for some $f\in M_{n}(\Z(G)),$ with action given by (\ref{E:algactionintro}) and with $\lambda(f)$ injective (implicit in \cite{Me5}),
 \item $X$ is a profinite group, $\mu=m_{X},$ where $G\actson X$ by automorphisms, and $G\actson X$ has a dense homoclinic group (implicit in \cite{GabSew}).\label{I:GabSewintro}
 \end{enumerate}
 Thus in cases $(\ref{I:finpresintro})-(\ref{I:GabSewintro}),$ we have that for every closed, normal, $G$-invariant subgroup $Y\subseteq X,$
 \[h_{(\sigma_{i})_{i},m_{X}}(X/Y:X,G)=h_{(\sigma_{i})_{i},\topo}(X/Y:X,G).\]
 \end{prop}

    Consequently, we will see in  \ref{C:dqintro} that Theorem \ref{T:mainintro} recovers all previously known results on equality of topological and measure-theoretic entropy in the nonamenable case. In the ergodic case we can say more about the entropy of $G\actson (X,m_{X}),$ and we show that the local and doubly empirical model-measure sofic entropy introduced in \cite{AustinAdd} agrees with the topological entropy and measure-theoretic entropy previously defined in \cite{Bow},\cite{KLi}. We use $h_{(\sigma_{i})_{i},\mu}^{lde}(X,G)$ for the local and doubly empirical convergence model-measure sofic entropy defined in \cite{AustinAdd} (there it is called doubly quenched model-measure sofic entropy and denoted $h^{dq}_{\Sigma}(\mu)$).

\begin{thm}\label{T:dqintroergcase} Let $G$ be a countable, discrete, sofic group with sofic approximation $\sigma_{i}\colon G\to S_{d_{i}}.$ Let $X$ be a compact, metrizable group with $G\actson X$ by automorphisms. If $G\actson (X,m_{X})$ is ergodic, and if there exists a sequence $\mu_{i}\in \Prob(X^{d_{i}})$ with $\mu_{i}\to^{lw^{*}}m_{X},$ then
\[h_{(\sigma_{i})_{i},m_{X}}^{lde}(X,G)=h_{(\sigma_{i})_{i},m_{X}}(X,G)=h_{(\sigma_{i})_{i},\topo}(X,G).\]
\end{thm}

Theorem \ref{T:dqintroergcase} says, under the assumption of ergodicity, once one knows the existence of some sequence of measures which locally weak$^{*}$ converge to the Haar measure, then we are able to promote this sequence to  a sequence of measures which locally and doubly empirically converge to the Haar measure and have good enough separation properties to compute the local and doubly empirical model-measure entropy. 

	Let us briefly outline the proof of Theorem \ref{T:dqintroergcase}. Consider a sequence $\mu_{i}\in \Prob(X^{d_{i}})$ which locally weak$^{*}$ converges to the Haar measure. One first makes a minor observation that our ergodicity assumption shows that in fact  $\mu_{i}$ locally and doubly empirically converges to the Haar measure (this is because ergodicity and weak mixing are equivalent for actions on compact groups). We consider a convolution of $\mu_{i}$ with the uniform measure of a well separated subset of $S_{i}\subseteq X^{d_{i}},$ where each element of $S_{i}$ is a good microstate for the topological entropy.  As is typical with convolution arguments, this new sequence inherits the good properties of each measure. This new sequence of measures has the separation properties that $u_{S_{i}}$ has, and Austin's results in \cite{AustinAdd} on behavior under products show  that the convolution of $\mu_{i}$ with the uniform measure of $S_{i}$ still locally and doubly empirically converges to the Haar measure. The fact that  $\mu_{i}*u_{S_{i}}$ has good separation properties, and that $\mu_{i}*u_{S_{i}}\to^{lde}m_{X},$ allows us to show the equality of all the entropies involved. We mention that this argument may be regarded as an extension of Berg's argument in \cite{Berg} to the nonamenable setting.
	
	Using Proposition \ref{P:mainexamples}, we have the following corollary which  reproves all previously known results regarding equality of topological and measure-theoretic entropy for actions on compact groups. Moreover, we can also equate these entropies to Austin's local and doubly empirical convergence model-measure entropy.
\begin{cor}\label{C:dqintro} Let $G$ be a countable, discrete, sofic group with sofic approximation $\sigma_{i}\colon G\to S_{d_{i}}.$ Let $X$ be a compact, metrizable group with $G\actson X$ by automorphisms. Suppose that one of the following cases hold:
\begin{enumerate}[(i):]
\item $X=X_{f}$ for some $f\in M_{m,n}(\Z(G))$ with $\lambda(f)$ having dense image,
\item $X=X_{f}$ for some $f\in M_{n}(\Z(G))$ with $\lambda(f)$ injective,
\item $X$ is a profinite group and $G\actson X$ has a dense homoclinic group,
\end{enumerate}
then
\[h_{(\sigma_{i})_{i},m_{X}}^{lde}(X,G)=h_{(\sigma_{i})_{i},m_{X}}(X,G)=h_{(\sigma_{i})_{i},\topo}(X,G).\]
\end{cor}

The examples listed in Corollary \ref{C:dqintro} constitute many of the main examples where sofic entropy has been computed and Corollary \ref{C:dqintro} computes the local and doubly empirical convergence model-measure sofic entropy introduced in \cite{AustinAdd} for each of these examples. Hence, Corollary \ref{C:dqintro} brings the computation of local and doubly empirical convergence model-measure sofic entropy up to date with that of the entropy defined by \cite{KLi},\cite{Bow}. Moreover we are able to do this in a relatively painless manner without  having to essentially reproduce the work in \cite{Me5},\cite{GabSew},\cite{KLi},\cite{BowenLi},\cite{BowenEntropy}. For example, as a consequence of Corollary \ref{C:dqintro}, our previous results connecting Fuglede-Kadison determinants and sofic measure entropy (namely Theorem 1.1 of \cite{Me5}) now hold with $h_{(\sigma_{i})_{i},m_{X_{f}}}(X_{f},\Gamma)$ replaced with $h_{(\sigma_{i})_{i},m_{X_{f}}}^{lde}(X,\Gamma)$ throughout.

We make a few remarks about the organization of the paper. Section \ref{S:reform} reformulates Austin's local and  empirical convergence, and his local and doubly empirical convergence model-measure entropy, so that one does not have to work with invariant measures on a shift space, but only with invariant measures on any compact $G$-space. In this section, we also prove Proposition \ref{P:mainexamples} and explains how it implicitly follows from \cite{Bow},\cite{Me5},\cite{GabSew},\cite{AustinAdd}. In Section \ref{S:equality}, we prove Theorems \ref{T:mainintro} and \ref{T:dqintroergcase} (which are the main theorems in the paper). The proofs in this section are basically convolution arguments combined with simple applications of Fubini's theorem.

\textbf{Acknowledgments.} I am indebted to Lewis Bowen and Brandon Seward for invaluable conversations. I thank Bingbing Liang for finding many typographical errors in a previous version of the paper. Part of this work was done while I was visiting the University of California, Los Angeles and I thank the University of California, Los Angeles for their hospitality and providing a stimulating environment in which to work.

\section{Reformulation of Local and Doubly Empirical Convergence}\label{S:reform}

In \cite{AustinAdd}, Austin defined the notion of local and empirical convergence (there called ``quenched convergence'') for invariant measures of $G\actson X^{G}$ where $G$ is a countable, discrete group, $X$ is a compact metrizable space, and the action is by left shifts. In this section, we explain how to modify the definition so that it works for invariant measures for an arbitrary action $G\actson Y$ where $Y$ is any compact, metrizable space. We  begin by recalling some necessary definitions and notation. For a set $A$ and $n\in\N,$ we identify $A^{n}$ with all functions $\{1,\dots,n\}\to A.$ For a finite set $A,$ we use $u_{A}$ for the uniform measure on $A,$ if $A=\{1,\dots,n\}$ for some $n\in\N$ we use $u_{n}$ instead of $u_{\{1,\dots,n\}}.$ We let $S_{n}$ be the group of all bijections $\{1,\dots,n\}\to \{1,\dots,n\}.$
\begin{defn}\label{D:sofic} Let $G$ be a countable, discrete group. A sequence of maps $\sigma_{i}\colon G\to S_{d_{i}}$ is said to be a \emph{sofic approximation} if:
\begin{itemize}
\item for every $g,h\in G$ we have $u_{d_{i}}(\{j:\sigma_{i}(g)\sigma_{i}(h)(j)=\sigma_{i}(gh)(j)\})\to 1,$
\item for every $g\in G\setminus\{e\}$ we have $u_{d_{i}}(\{j:\sigma_{i}(g)(j)=j\})\to 0.$
\end{itemize}
We say that $G$ is \emph{sofic} if it has a sofic approximation.

\end{defn}

The first item in the definition of sofic approximation may be thought of as saying that $G$ has an ``almost action" on the finite set $\{1,\dots,d_{i}\}$ (in the sense that the points on which the action axiom $\sigma_{i}(gh)(j)=\sigma_{i}(g)\sigma_{i}(h)(j)$ fails is neglible), the second item says that this action is ``almost free." Thus, sofic groups may be regarded as a perturbation of finite groups, because finite groups are precisely the groups which have free actions on finite sets.

We recall the definition of the topological and measure-theoretic microstates space. Suppose we are given a countable, discrete group $G$ and a compact, metrizable space $X$ with $G\actson X$ by homeomorphisms. A pseudometric $\rho$ on $X$ is said to be dynamically generating if for all $x\ne y$ with $x,y\in X,$ there is a $g\in G$ with $\rho(gx,gy)>0.$
If $\rho$ is as above and $n\in \N,$ we define $\rho_{2}$ on $X^{n}$ by
\[\rho_{2}(x,y)^{2}=\frac{1}{n}\sum_{j=1}^{n}\rho(x(j),y(j))^{2}.\]

\begin{defn}\label{D:topmicrospace}Let $G$ be a countable, discrete, sofic group with sofic approximation $\sigma_{i}\colon G\to S_{d_{i}}.$ Suppose that $X$ is a compact, metrizable space with $G\actson X$ by homeomorphisms, and that $\rho$ is a dynamically generating pseudometric on $X$. For a finite $F\subseteq G,$ and a $\delta>0,$ we let $\Map(\rho,F,\delta,\sigma_{i})$ be all $\phi\in X^{d_{i}}$ so that
\[\rho_{2}(g\phi,\phi\circ \sigma_{i}(g))<\delta, \mbox{ for all $g\in F.$}\]
\end{defn}

We think of $\Map(\rho,F,\delta,\sigma_{i})$ as a space of ``topological microstates'' in the sense that $\phi\in \Map(\rho,F,\delta,\sigma_{i})$ gives a finitary model of $G\actson X$ (i.e. it is approximately equivariant with approximate being measured in the topology on $X$). For a compact, metrizable space $X,$  we let $\Prob(X)$ be the space of Borel probability measures on $X.$ For a finite $L\subseteq C(X),$ and $\delta>0,$ set
\[U_{L,\delta}(\mu)=\bigcap_{f\in L}\left\{\nu\in \Prob(X):\left|\int_{X}f\,d\mu-\int_{X}f\,d\nu\right|<\delta\right\}.\]
Note that $U_{L,\delta}(\mu)$ ranging over all $L,\delta$ gives a basis of neighborhoods of $\mu$ in the weak-$^{*}$ topology on $\Prob(X).$

\begin{defn}\label{D:measmicrospace}Let $G$ be a countable, discrete, sofic group with sofic approximation $\sigma_{i}\colon G\to S_{d_{i}}.$ Suppose that $X$ is a compact, metrizable space, that $G\actson X$ by homeomorphisms, and that $\rho$ is a dynamically generating pseudometric. For finite $F\subseteq G,L\subseteq C(X),$ and $\delta>0$ set
\[\Map_{\mu}(\rho,F,L,\delta,\sigma_{i})=\{\phi\in \Map(\rho,F,\delta,\sigma_{i}):\phi_{*}(u_{d_{i}})\in U_{L,\delta}(\mu)\}.\]
\end{defn}
We think of $\Map_{\mu}(\rho,F,L,\delta,\sigma_{i})$ as a space of ``measure-theoretic microstates'' for $G\actson (X,\mu).$

Given a compact, metrizable space $X$ and a countable discrete group with $G\actson X$ by homeomorphisms, we let $\Prob_{G}(X)$ be the set of all $\mu\in \Prob(X)$ which are $G$-invariant.

\begin{defn}\label{D:quenched} Let $G$ be a countable, discrete, sofic group with sofic approximation $\sigma_{i}\colon G\to S_{d_{i}}.$ Let $X$ be a compact, metrizable space with $G\cc X$ by homeomorphisms and let $\rho$ be a dynamically generating pseudometric on $X.$ We say that a sequence $\mu_{i}\in \Prob(X^{d_{i}})$ locally and empirically converges to $\mu\in \Prob_{G}(X)$ if for all finite $F\subseteq G, L\subseteq C(X),$ and $\delta>0$ we have
\[\mu_{i}\left(\Map_{\mu}(\rho,F,L,\delta,\sigma_{i})\right)\to 1,\]
\[\min_{f\in L}u_{d_{i}}\left(\left\{1\leq j\leq d_{i}:\left|\int_{X}f\,d\mu-\int_{X^{d_{i}}}f(\phi(j))\,d\mu_{i}(\phi)\right|<\delta\right\}\right)\to 1.\]
We write $\mu_{i}\to^{le}\mu$ to mean that $\mu_{i}$ locally and empirically converges to $\mu.$
\end{defn}

We give a proposition which states a few equivalent ways to say that a sequence $\mu_{i}\in \Prob(X^{d_{i}})$ locally and empirically converges to $\mu.$ This proof is simple and  is left to the reader, we will not actually need it in this paper and mainly state it to give the reader added intuition as to the meaning of local and empirical convergence. Let $X$ be a set and $n\in \N,$ define, for $1\leq j\leq n,$ a map $\mathcal{E}_{j}\colon X^{n}\to X$ by $\mathcal{E}_{j}(x)=x(j).$ 

\begin{prop}Let $G$ be a countable, discrete, sofic group with sofic approximation $\sigma_{i}\colon G\to S_{d_{i}}.$ Let $X$ be a compact, metrizable space with $G\cc X$ by homeomorphisms and fix a $\mu\in \Prob_{G}(X).$ Fix a dynamically generating pseudometric $\rho$ on $X.$ For a given sequence $\mu_{i}\in \Prob(X^{d_{i}}),$  the following are equivalent:
\begin{enumerate}[(i)]
\item $\mu_{i}\to^{le}\mu,$\\
\item For any neighborhood $\mathcal{O}$ of $\mu$ in the weak-$^{*}$ topology, for any neighborhood $\mathcal{V}$ of the diagonal in $X\times X$,  and any $g\in G$ we have
\[u_{d_{i}}(\{j:(\mathcal{E}_{j})_{*}\mu_{i}\in \mathcal{O}\})\to 1,\]
\[\mu_{i}(\{x:x_{*}(u_{d_{i}})\in \mathcal{O}\})\to 1,\]
\[\mu_{i}\otimes u_{d_{i}}(\{(x,j):(gx(j),x(\sigma_{i}(g)(j)))\in \mathcal{V}\})\to 1.\]

\item For any neighborhood $\mathcal{O}$ of $\mu$  in the weak-$^{*}$ topology and for any $g\in G$ we have
\[u_{d_{i}}(\{j:(\mathcal{E}_{j})_{*}\mu_{i}\in \mathcal{O}\})\to 1,\]
\[\mu_{i}(\{x:x_{*}(u_{d_{i}})\in \mathcal{O}\})\to 1,\]
\[\int \rho(gx(j),x(\sigma_{i}(g)(j)))\,d\mu_{i}\otimes u_{d_{i}}(x,j)\to 0.\]
\end{enumerate}

\end{prop}

A priori, the definition of local and empirical convergence depends upon a choice of dynamically generating pseudometric (the preceding proposition actually shows independence but we will give a different proof). We will ignore this  because we will show that our definition of  local and empirical convergence agrees with the definition of quenched convergence of \cite{AustinAdd}. It is clear that the definition of quenched convergence (resp. doubly quenched convergence) given in \cite{AustinAdd} only depends upon the topology of $X,$ but in Corollary 5.18 of \cite{AustinAdd} it is in fact shown that it only depends upon the measure-theoretic nature of $G\actson (X,\mu).$  We now restate the definition of quenched convergence defined in \cite{AustinAdd}. Let $G$ be a countable, discrete, sofic group $G$ with a sofic approximation $\sigma_{i}\colon G\to S_{d_{i}}.$ For $x\in X^{d_{i}},$ we define $\phi_{x}\colon \{1,\dots,d_{i}\}\to X^{G}$ by
\[(\phi_{x}(j))(g)=x(\sigma_{i}(g)^{-1}(j)),\mbox{ for $1\leq j\leq d_{i},g\in G$}.\]
\begin{defn}[\cite{AustinAdd} Definition 5.3] Let $G$ be a countable, discrete, sofic group with sofic approximation $\sigma_{i}\colon G\to S_{d_{i}}.$ Let $X$ be a compact, metrizable, space and let $G\cc X^{G}$ by the left shift action given by $(gx)(h)=x(g^{-1}h)$.  Fix a $\mu\in \Prob_{G}(X^{G}).$ We say that a sequence $\mu_{i}\in \Prob(X^{d_{i}})$\emph{ quenched converges to $\mu$} (in the sense of \cite{AustinAdd}) if given any neighborhood $\mathcal{O}$ of $\mu$ in the weak$^{*}$ topology on $\Prob(X^{G})$ we have
\[\mu_{i}(\{x:(\phi_{x})_{*}(u_{d_{i}})\in \mathcal{O}\})\to 1,\]
\[u_{d_{i}}\left(\{j:(\mathcal{E}_{j})_{*}(\mu_{i})\in \mathcal{O}\}\right)\to 1.\]

\end{defn}

We show that our notion of local and empirical convergence is related to that of quenched convergence of \cite{AustinAdd}. We need the following lemma (which will also be used later).

\begin{lem}\label{L:transferlemma} Let $G$ be a countable discrete sofic group with sofic approximation $\sigma_{i}\colon G\to S_{d_{i}}.$ Let $X$ be a compact, metrizable space with $G\cc X$ by homeomorphisms, and let $\rho$ be a dynamically generating pseudometric on $X.$  Define $\Psi\colon X\to X^{G}$ by $\Psi(x)(g)=g^{-1}x.$

(i): For any neighborhood $\mathcal{O}$ of $\Psi_{*}(\mu)$ in the weak-$^{*}$ topology on $\Prob(X^{G}),$ there exist finite $F\subseteq G,L\subseteq C(X),$ and $\delta>0$ so that for all sufficiently large $i,$ and all $x\in \Map_{\mu}(\rho,F,L,\delta,\sigma_{i}),$ we have $(\phi_{x})_{*}(u_{d_{i}})\in \mathcal{O}.$

(ii): Given any finite $F\subseteq G,L\subseteq C(X),$ and $\delta>0,$ there is a neighborhood $\mathcal{O}$ of $\Psi_{*}(\mu)$ in the weak$^{*}$-topology so that if $i$ is sufficiently large, then for every $x\in X^{d_{i}}$ with $(\phi_{x})_{*}(u_{d_{i}})\in \mathcal{O},$ we have $x\in \Map_{\mu}(\rho,F,\delta,L,\sigma_{i}).$

\end{lem}

\begin{proof}

 (i): For $g\in G,$ let $\iota_{g}\colon C(X)\to C(X^{G})$ be the map defined by $\iota_{g}(f)(x)=f(x(g)).$ Assume we are given a finite $E\subseteq G,$ a $(f_{g})_{g\in E}\in C(X)^{E},$ and an $\varepsilon>0$. It is enough to handle the case that
\[\mathcal{O}=\left\{\nu\in \Prob(X^{G}):\left|\int_{X^{G}}\prod_{g\in E}\iota_{g}(f_{g})\,d\nu-\int_{X^{G}}\prod_{g\in E}\iota_{g}(f_{g})\,d\Psi_{*}(\mu)\right|<\varepsilon\right\}.\]
This is because the collection of all $\prod_{g\in E}\iota_{g}(f_{g})$ ranging over all finite $E\subseteq G$ and $(f_{g})_{g\in E}\in C(X)^{E}$ has dense linear span, by the Stone-Weierstrass Theorem.  Define $\xi\in C(X)$ by
\[\xi(x)=\prod_{g\in E}f_{g}(g^{-1}x).\]

Let $C=\max_{g\in E}\|f_{g}\|.$ We may choose a finite $K\subseteq G$ and a $\kappa>0$ so that if $x,y\in X$ and $\max_{h\in K}\rho(hx,hy)<\kappa,$ then $\max_{g\in E}|f_{g}(x)-f_{g}(y)|<\frac{\varepsilon}{3|E|C^{|E|-1}}.$
Suppose that $F$ is a finite subset of $G$ containing
\[(K\cup K^{-1}\cup\{e\})(E\cup E^{-1}\cup \{e\}).\]
Now choose $\delta\in (0,\varepsilon/3)$ sufficiently small so that $\frac{9|K|||E|\delta^{2}}{\kappa^{2}}<\frac{\varepsilon}{6C^{|E|}}.$ Note that for $x\in \Map_{\mu}(\rho,F,\{\xi\},\delta,\sigma_{i}),$ and $h\in K,g\in E$ we have
\begin{align*}
\rho_{2}(h\cdot(x\circ \sigma_{i}(g)^{-1}),hg^{-1}x)&\leq \delta+\rho_{2}(h\cdot (x\circ \sigma_{i}(g)^{-1}),x\circ \sigma_{i}(hg^{-1}))\\
&\leq 2\delta+\rho_{2}(x\circ \sigma_{i}(h)\sigma_{i}(g^{-1}),x\circ \sigma_{i}(hg^{-1})).
\end{align*}
By soficity and the fact that $\rho$ is bounded, the estimate above shows that for all sufficiently large $i$ we have
\[\rho_{2}(h\cdot(x\circ \sigma_{i}(g)^{-1}),hg^{-1}x)<3\delta.\]
Suppose $x\in \Map_{\mu}(\rho,F,L,\delta,\sigma_{i}),$ and let
\[J=\bigcap_{h\in K,g\in E}\{1\leq j\leq d_{i}:\rho(hg^{-1}x(j),hx(\sigma_{i}(g)^{-1}(j)))<\kappa\}.\]
By Chebyshev's inequality 
\[u_{d_{i}}(J^{c})\leq \frac{9|K||E|\delta^{2}}{\kappa^{2}}<\frac{\varepsilon}{6C^{|E|}},\]
for all sufficiently large $i.$
For $x\in \Map_{\mu}(\rho,F,\{\xi\},\delta,\sigma_{i})$ we  have
\[\left|\frac{1}{d_{i}}\sum_{j=1}^{d_{i}}\prod_{g\in E}f_{g}([g^{-1}x(j)])-\int_{X}\xi(x)\,d\mu(x)\right|<\varepsilon/3,\]
and for all sufficiently large $i$
\begin{align*}
\left|\frac{1}{d_{i}}\sum_{j=1}^{d_{i}}\prod_{g\in E}f_{g}([g^{-1}x(j)])-\frac{1}{d_{i}}\sum_{j=1}^{d_{i}}\prod_{g\in E}f_{g}([x(\sigma_{i}(g)^{-1}(j))])\right|&\leq 2C^{|E|}u_{d_{i}}(J^{c})\\
&+\frac{1}{d_{i}}\sum_{j\in J}\left|\prod_{g\in E}f_{g}([g^{-1}x(j)])-\prod_{g\in E}f_{g}([x(\sigma_{i}(g)^{-1}(j))])\right|\\
&\leq \frac{\varepsilon}{3}\\
&+\frac{1}{d_{i}}\sum_{j\in J}\sum_{g\in E}C^{|E|-1}|f_{g}(g^{-1}x(j))-f_{g}(x(\sigma_{i}(g)^{-1}(j)))|\\
&\leq \frac{2\varepsilon}{3}.
\end{align*}
Hence it follows that  for all sufficiently large $i,$
\[\left|\int_{X}\xi(x)\,d\mu(x)-\frac{1}{d_{i}}\sum_{j=1}^{d_{i}}\prod_{g\in E}f_{g}(x(\sigma_{i}(g)^{-1}(j)))\right|<\varepsilon.\]
As
\[\int_{X}\xi(x)\,d\mu(x)=\int_{X^{G}}\prod_{g\in E}\iota_{g}(f_{g})\,d\Psi_{*}\mu,\]
we see that $(\phi_{x})_{*}(u_{d_{i}})\in \mathcal{O}.$

(ii): Let $M$ be the  diameter of $(X,\rho).$ We may choose $f_{1},\dots,f_{k}\in C(X)$ and $\eta>0$ so that if $x,y\in X$ and
\[\max_{1\leq l\leq k}|f_{l}(x)-f_{l}(y)|<\eta,\]
then $\rho(x,y)<\delta/2.$ Fix $\varepsilon\in (0,\delta)$ sufficiently small so that $\frac{4\varepsilon^{2}}{\eta^{2}}k^{2}<\frac{3\delta^{2}}{4M^{2}}.$  For $g\in F,$ and $1\leq l\leq k,$ define $H_{l,g}\in C(X^{G})$ by
\[H_{l,g}(x)=|f_{l}(gx(e))-f_{l}(x(g^{-1}))|.\]
For every $g\in F,$ and $1\leq l\leq k,$ we have
\[\int H_{l,g}(x)\,d\Psi_{*}\mu(x)=0.\]
Set
\[\mathcal{O}=\bigcap_{f\in L}\left\{\nu\in \Prob(X^{G}):\left|\int_{X^{G}}\iota_{e}(f)\,d\nu-\int_{X^{G}}\iota_{e}(f)\,d\Psi_{*}(\mu)\right|<\varepsilon\right\}\cap \bigcap_{\substack{1\leq l\leq k,\\ g\in F}}\left\{\nu\in \Prob(X^{G}):\int_{X^{G}}H_{l,g}\,d\nu<\varepsilon\right\}.\]
Suppose that $x\in X^{d_{i}}$ and $(\phi_{x})_{*}(u_{d_{i}})\in \mathcal{O},$ we then clearly have
\[\left|\frac{1}{d_{i}}\sum_{j=1}^{d_{i}}f(x(j))-\int f\,d\mu\right|<\delta.\]
Fix $g\in F$ and set
\[J=\bigcap_{l=1}^{k}\left\{1\leq j\leq d_{i} :H_{l,g}(\phi_{x}(j))<\eta\right\}\cap \{1\leq j\leq d_{i}:\sigma_{i}(g^{-1})^{-1}(j)=\sigma_{i}(g)(j)\}.\]
By soficity, for all large $i$ we have
\[u_{d_{i}}(J^{c})\leq k\frac{2\varepsilon}{\eta}.\]
For $g\in F$ and $j\in J$ we have
\[\max_{1\leq l\leq k}|f_{l}(gx(j))-f_{l}(x(\sigma_{i}(g^{-1})^{-1}(j)))|=\max_{1\leq l\leq k}H_{l,g}(\phi_{x}(j))<\eta,\]
so
\[\rho(gx(j),x(\sigma_{i}(g)(j)))=\rho(gx(j),x(\sigma_{i}(g^{-1})^{-1}(j)))<\frac{\delta}{2}.\]
Thus for all sufficiently large $i$
\[\rho_{2}(gx,x\circ \sigma_{i}(g))^{2}\leq \frac{4\varepsilon^{2}}{\eta^{2}}M^{2}k^{2}+\frac{\delta^{2}}{4}<\delta^{2},\]
so $x\in \Map_{\mu}(\rho,F,L,\delta,\sigma_{i}).$

\end{proof}

\begin{lem}\label{L:dqequvialkdghalk} Let $G$ be a countable discrete sofic group with sofic approximation $\sigma_{i}\colon G\to S_{d_{i}}.$ Let $X$ be a compact, metrizable space with $G\cc X$ by homeomorphisms. Suppose that $\rho$ is a dynamically generating pseudometric on $X.$ Define $\Psi\colon X\to X^{G}$ by $\Psi(x)(g)=g^{-1}x.$ Then $\mu_{i}$ locally and empirically converges to $\mu$ in the sense of definition \ref{D:quenched} if and only if $\mu_{i}$ quenched converges to $\Psi_{*}\mu$ in the sense of \cite{AustinAdd}.

\end{lem}

\begin{proof} Let $M$ be the diameter of $(X,\rho).$ First suppose that $\mu_{i}$ quenched converges to $\Psi_{*}\mu$ in the sense of \cite{AustinAdd}.
It is not hard to argue from Lemma \ref{L:transferlemma} that
\[\mu_{i}(\Map_{\mu}(\rho,F,L,2\delta,\sigma_{i}))\to 1.\]
As
\[\int_{X^{d_{i}}}f(\phi_{x}(j)(e))\,d\mu_{i}(x)=\int_{X^{d_{i}}}f(x(j))\,d\mu_{i}(x),\]
\[\int_{X^{G}}f(x(e))\,d\Psi_{*}\mu(x)=\int_{X}f(x)\,d\mu(x),\]
it follows easily from the fact that $\mu_{i}$ quenched converges to $\Psi_{*}\mu$ in the sense of \cite{AustinAdd} that
\[\min_{f\in L}u_{d_{i}}\left(\left\{j:\left|\int_{X^{d_{i}}}f(x(j))-\int_{X}f\,d\mu\right|<\delta\right\}\right)\to 1.\]
Thus $\mu_{i}$ locally and empirically converges to $\mu$ in the sense of Definition \ref{D:quenched}.

Conversely, suppose that $\mu_{i}$ locally and empirically converges to $\mu$ in the sense of Definition \ref{D:quenched}. It is not hard to show from Lemma \ref{L:transferlemma} that for any weak$^{*}$ neighborhood $\mathcal{O}$ of $\Psi_{*}\mu$ in $\Prob(X^{G})$ we have
\[\mu_{i}(\{x:(\phi_{x})_{*}(u_{d_{i}}) \in\mathcal{O}\})\to 1.\]
For $g\in G,$ let $\iota_{g}\colon C(X)\to C(X^{G})$ be given by
\[\iota_{g}(f)(x)=f(x(g)).\]
Let
\[A=\bigcup_{\textnormal{ finite} E\subseteq G}\left\{\prod_{g\in E}\iota_{g}(f_{g}):(f_{g})_{g\in E}\in C(X)^{E}\right\},\]
and observe that $A$ has dense linear span in $C(X^{G}).$  Fix a $\delta>0,$ a finite $E\subseteq G,$ and a $(f_{g})_{g\in E}\in C(X)^{E}.$ Since $A$ has dense linear span in $C(X^{G}),$ to show that $\mu_{i}$ quenched converges to $\Psi_{*}\mu$ in the sense of \cite{AustinAdd} it is enough to show that
\[u_{d_{i}}\left(\left\{j:\left|\int_{X^{d_{i}}}\prod_{g\in E}f_{g}(x(\sigma_{i}(g)^{-1}(j)))\,d\mu_{i}(x)-\int_{X}\prod_{g\in E}f_{g}(g^{-1}x)\,d\mu(x)\right|>\delta\right\}\right)\to 0.\]
Since $\mu_{i}$ locally and empirically converges to $\mu$ in the sense of Definition \ref{D:quenched}, it is clear that
\[u_{d_{i}}\left(\left\{j:\left|\int_{X^{d_{i}}}\prod_{g\in E}f_{g}(g^{-1}(x(j)))\,d\mu_{i}(x)-\int_{X}\prod_{g\in E}f_{g}(g^{-1}x)\,d\mu(x)\right|>\frac{\delta}{2}\right\}\right)\to 0.\]
So it suffices to show
\[u_{d_{i}}\left(\left\{j:\left|\int_{X^{d_{i}}}\prod_{g\in E}f_{g}(g^{-1}(x(j)))\,d\mu_{i}(x)-\int_{X^{d_{i}}}\prod_{g\in E}f_{g}(x(\sigma_{i}(g)^{-1}(j)))\,d\mu_{i}(x)\right|>\frac{\delta}{2}\right\}\right)\to 0,\]
and to show this it is enough to establish
\begin{equation}\label{E:bigreductionalsdjafl}
\frac{1}{d_{i}}\sum_{j=1}^{d_{i}}\int_{X^{d_{i}}}\left|\prod_{g\in E}f_{g}(g^{-1}(x(j)))-\prod_{g\in E}f_{g}(x(\sigma_{i}(g)^{-1}(j)))\right|\,d\mu_{i}(x)\to 0.
\end{equation}
Let $C=\max_{g\in E}\|f_{g}\|,$ then
\[\frac{1}{d_{i}}\sum_{j=1}^{d_{i}}\int_{X^{d_{i}}}\left|\prod_{g\in E}f_{g}(g^{-1}(x(j)))-\prod_{g\in E}f_{g}(x(\sigma_{i}(g)^{-1}(j)))\right|\,d\mu_{i}(x)\leq \]
\[\frac{1}{d_{i}}\sum_{j=1}^{d_{i}}\sum_{g\in E}C^{|E|-1}\int_{X^{d_{i}}}\left|f_{g}(g^{-1}(x(j)))-f_{g}(x(\sigma_{i}(g)^{-1}(j)))\right|\,d\mu_{i}(x)\]
and so it is enough to establish that, for every $g\in E,$ we have
\begin{equation}\label{E:bigreduction1dljgaljalgj}
\frac{1}{d_{i}}\sum_{j=1}^{d_{i}}\int_{X^{d_{i}}}|f_{g}(g^{-1}(x(j)))-f_{g}(x(\sigma_{i}(g)^{-1}(j)))|\,d\mu_{i}(x)\to 0.
\end{equation}

	Fix $g\in E$ and an $\varepsilon>0.$ Because $\rho$ is dynamically generating, we may find an $\eta>0$ and a finite $F\subseteq G$ so that if $x,y\in X$ and $\max_{h\in F}\rho(hx,hy)<\eta,$ then $\max_{g\in E}|f_{g}(x)-f_{g}(y)|<\varepsilon.$ Let
	\[A_{i}=\bigcap_{h\in F}\{(x,j):\rho(hg^{-1}(x(j)),hx(\sigma_{i}(g)^{-1}(j)))<\eta\}.\]
For $h\in F,g\in E,x\in X^{d_{i}}$ we have
\[\rho_{2}(hg^{-1}x,h(x\circ \sigma_{i}(g)^{-1}))\leq \rho_{2}(hg^{-1}x,x\circ \sigma_{i}(hg^{-1}))+\rho_{2}(x\circ \sigma_{i}(h)\sigma_{i}(g)^{-1},x\circ \sigma_{i}(hg^{-1}))+\rho_{2}(hx,x\circ \sigma_{i}(h)).\]
By soficity,
\[\sup_{x\in X^{d_{i}}}\rho_{2}(x\circ \sigma_{i}(h)\sigma_{i}(g)^{-1},x\circ \sigma_{i}(hg^{-1})\to 0.\]
Hence
\[\limsup_{i\to\infty}\mu_{i}\otimes u_{d_{i}}(A_{i}^{c})\leq \frac{1}{\eta^{2}}\limsup_{i\to\infty}\int_{X^{d_{i}}}\sum_{h\in F} (\rho_{2}(hg^{-1}x,x\circ \sigma_{i}(hg^{-1}))+\rho_{2}(hx,x\circ \sigma_{i}(h)))^{2}\,d\mu_{i}(x)= 0,\]
because $\mu_{i}$ locally and empirically converges to $\mu$ in the sense of Definition \ref{D:quenched}. It follows that
\[\limsup_{i\to\infty}\frac{1}{d_{i}}\sum_{j=1}^{d_{i}}\int_{X^{d_{i}}}|f_{g}(g^{-1}(x(j)))-f_{g}(x(\sigma_{i}(g)^{-1}(j)))|\,d\mu_{i}(x)=\]
\[\limsup_{i\to\infty}\int_{A_{i}}|f_{g}(g^{-1}(x(j)))-f_{g}(x(\sigma_{i}(g)^{-1}(j)))|d(\mu_{i}\otimes u_{d_{i}})(x,j)\leq \varepsilon.\]
Since $\varepsilon$ is arbitrary, we have shown  (\ref{E:bigreduction1dljgaljalgj}).

\end{proof}

Because of the above lemma, and the fact that context will make it clear which sofic approximation we are using, if $\sigma_{i},X,\mu_{i}$ are as in the theorem we use $\mu_{i}\to^{le}\mu$ to mean that $\mu_{i}$ locally and empirically converges to $\mu$ in the sense of Definition \ref{D:quenched}. We say that $\mu_{i}$ \emph{locally and doubly empirically converges}to $\mu,$ and write $\mu_{i}\to^{lde}\mu,$ if $\mu_{i}\otimes \mu_{i}\to^{le}\mu\otimes \mu.$  By Corollary 5.18 of \cite{AustinAdd}, it follows that the existence of a sequence of measures which locally and empirically converge to $\mu$ is independent of the topological model. Note that Lemma \ref{L:dqequvialkdghalk} implies an obvious analogue of Lemma \ref{L:dqequvialkdghalk} for doubly quenched convergence. By the results of \cite{AustinAdd}, it follows that existence of a sequence of measures which locally and doubly empirically converge to $\mu$ is independent of the topological model and is preserved under measure-theoretic factor maps.

We reformulate a few of the results from \cite{AustinAdd} in our framework. Given sets $X,Y$ a natural number $n,$ and $\phi\in X^{n},\psi\in Y^{n},$ we define $\phi\otimes \psi\in (X\times Y)^{n}$ by $(\phi\otimes\psi)(j)=(\phi(j),\psi(j))$ for $1\leq j\leq n.$

\begin{cor}\label{C:dqfacts}  Let $G$ be a countable discrete sofic group with sofic approximation $\sigma_{i}\colon G\to S_{d_{i}}.$ Let $X$ be a compact, metrizable space with $G\cc X$ by homeomorphisms, and let $\rho$ be a dynamically generating pseudometric on $X.$ Fix a $\mu\in \Prob_{G}(X)$ and a sequence $\mu_{i}\in \Prob(X^{d_{i}}).$ 
\begin{enumerate}[(i)]
\item\label{I:productdq} Assume that $\mu_{i}\to^{lde}\mu.$ Let $Y$ be a compact, metrizable space, let $G\actson Y$ by homeomorphisms and let $\nu\in \Prob_{G}(X).$ Then for any dynamically generating pseudometrics $\rho',\widetilde{\rho}$ on $X\times Y,Y,$  for any finite $F'\subseteq G,L'\subseteq C(X\times Y),$ and any $\delta'>0,$ there exists  finite $F\subseteq G,L\subseteq C(Y),$ and a $\delta>0$ so that
\[\lim_{i\to\infty}\inf_{\psi\in \Map_{\nu}(\widetilde{\rho},F,L,\delta,\sigma_{i})}\mu_{i}(\{\phi:\phi\otimes \psi\in \Map_{\mu\otimes \nu}(\rho',F',L',\delta',\sigma_{i})\})=1.\]
\item\label{I:wmdq} Assume that $G\actson (X,\mu)$ is ergodic. Suppose that 
\begin{itemize}
\item for any finite $F\subseteq G,$ and any $\delta>0,$
\[\mu_{i}(\Map(\rho,F,\delta,\sigma_{i}))\to 1,\]
\item for any weak-$^{*}$ neighborhood $\mathcal{O}$ of $\mu$ in $\Prob(X)$ we have
\[u_{d_{i}}(\{j:(\mathcal{E}_{j})_{*}\mu_{i}\in\mathcal{O}\})\to 1,\]
\end{itemize}
then $\mu_{i}\to^{le}\mu.$ If in addition $G\actson (X,\mu)$ is weakly mixing, then $\mu_{i}\to^{lde}\mu.$
\end{enumerate}

\end{cor}

\begin{proof} Once we know Lemmas \ref{L:dqequvialkdghalk} and \ref{L:transferlemma} these all follows from some of the main results in \cite{AustinAdd}. Specifically, part (\ref{I:productdq}) follows from Theorem A of \cite{AustinAdd} and part (\ref{I:wmdq}) follows from Corollary 5.7 and Lemma 5.15 of \cite{AustinAdd}.

\end{proof}

There is a slightly subtlety in part (\ref{I:wmdq}) of Corollary \ref{C:dqfacts}. If we replace
\[\mu_{i}(\Map(\rho,F,\delta,\sigma_{i}))\to 1\]
with the assumption that $\mu_{i}$ is asymptotically  supported on \emph{measure-theoretic microstates}, i.e. that for any finite $F\subseteq G,L\subseteq C(X),$ and $\delta>0$ we have
\[\mu_{i}(\Map_{\mu}(\rho,F,L,\delta,\sigma_{i}))\to 1,\]
then part (\ref{I:wmdq}) is a tautology. The important aspect of part (\ref{I:wmdq}) is that we instead require that $\mu_{i}$ is only asymptotically supported on  \emph{topological microstates}.

For later use we also give the following reformulation of Austin's dq-entropy. Let $X$ be a compact, metrizable space, let $\rho$ be a continuous pseudometric on $X$ and $\mu\in \Prob(X).$   We let $S_{\varepsilon,\delta}(\mu,\rho)$ be the smallest cardinality of $A\subseteq X$ whose $\varepsilon$-neighborhood with respect to $\rho$ has $\mu$-measure at least $1-\delta.$ We remark that the roles of $\varepsilon,\delta$ are switched in page $7$ of \cite{AustinAdd}, we have elected to do this to stay close the notation in \cite{KLi2}.
\begin{defn}\label{D:dqreformaldjgaj}  Let $G$ be a countable discrete sofic group with sofic approximation $\sigma_{i}\colon G\to S_{d_{i}}.$ Let $X$ be a compact metrizable space, let $G\actson X$ by homeomorphisms, and  $\mu\in \Prob_{G}(X).$ Suppose that $d_{i_{n}}$ is a subsequence of $d_{i},$ and that $\mu_{i_{n}}\in \Prob(X^{d_{i_{n}}})$ satisfies $\mu_{i_{n}}\to^{lde}\mu.$ Set
\[h_{(\sigma_{i})_{i}}((\mu_{i_{n}})_{n},\rho,\varepsilon;\delta)=\limsup_{n\to\infty}\frac{1}{d_{i_{n}}}\log S_{\varepsilon,\delta}(\mu_{i_{n}},\rho_{2}).\]
We then let
\[h_{(\sigma_{i})_{i}}^{lde}(\mu,\rho,\varepsilon;\delta)=\sup_{(\mu_{i_{n}})_{n}}h_{(\sigma_{i})_{i}}((\mu_{i_{n}})_{n},\rho,\varepsilon;\delta),\]
where the supremum is over all subsequence $d_{i_{n}}$ and all measures $\mu_{i_{n}}$ with $\mu_{i_{n}}\to^{lde}\mu.$ Now set
\[h_{(\sigma_{i})_{i},\mu}^{lde}(\rho)=\sup_{\varepsilon,\delta>0}h_{(\sigma_{i})_{i}}^{lde}(\mu,\rho,\varepsilon;\delta).\]
We call $h_{(\sigma_{i})_{i},\mu}^{lde}(\rho)$ the \emph{local and doubly empirical model-measure sofic entropy}.
\end{defn}
The following is proven exactly as in Lemma 4.4 of \cite{Li}, Lemma 6.12 of \cite{BowenGroupoid} and Lemma 3.9 of \cite{Me6}.

\begin{prop}\label{P:switchmetric} Let $G$ be a countable discrete sofic group with sofic approximation $\sigma_{i}\colon G\to S_{d_{i}}.$ Let $X$ be a compact metrizable space, let $G\actson X$ by homeomorphisms, and $\mu\in \Prob_{G}(X).$ Suppose we are given a dynamically generating pseudometric $\rho$ on $X.$ Let $(a_{g})_{g\in G}$ be positive real numbers with $\sum_{g}a_{g}=1$ and define $\widetilde{\rho}$ on $X$ by
\[\widetilde{\rho}(x,y)=\sum_{g\in G}a_{g}\rho(gx,gy).\]
Then $\widetilde{\rho}$ is a compatible metric on $X$ and
\[h_{(\sigma_{i})_{i},\mu}^{lde}(\rho)=h_{(\sigma_{i})_{i},\mu}^{lde}(\widetilde{\rho}).\]
\end{prop}
The following is clear from Lemma \ref{L:dqequvialkdghalk} and Proposition \ref{P:switchmetric}.
\begin{prop} Let $G$ be a countable discrete sofic group with sofic approximation $\Sigma=(\sigma_{i}\colon G\to S_{d_{i}}).$ Let $X$ be a compact metrizable space, let $G\actson X$ by homeomorphisms, and $\mu\in \Prob_{G}(X).$ The local and doubly empirical entropy $h_{(\sigma_{i})_{i},\mu}^{lde}(\rho)$ as defined in Definition \ref{D:dqreformaldjgaj} agrees with $h_{\Sigma}^{dq}(\mu)$ as defined in Section 6 of \cite{AustinAdd}. Hence $h_{(\sigma_{i})_{i},\mu}^{lde}(\rho)$ does not depend upon the dynamically generating pseudometric.
\end{prop}
Because of the preceding proposition, we will use $h_{(\sigma_{i})_{i},\mu}^{lde}(X,G)$  instead of $h_{(\sigma_{i})_{i},\mu}^{lde}(\rho)$ whenever $\rho$ is a dynamically generating pseudometric on $X.$

We will show in Section \ref{S:equality} that the existence of a sequence which locally and doubly empirically converges to the Haar measure implies equality of topological and measure-theoretic entropy for actions by automorphisms on compact groups. In order to illustrate the benefit of these results, we close this section by giving many examples of actions on compact groups for which there is a sequence of measures which locally and doubly empirically converge to the Haar measure. 

	We need to know the following simple fact about weak mixing for actions on compact groups. This fact is well known and we only include the proof for completeness. For Hilbert spaces $\mathcal{H},\mathcal{K}$ we use $B(\mathcal{H},\mathcal{K})$ for the space of bounded, linear operators $\mathcal{H}\to \mathcal{K}.$ For a Hilbert space $\mathcal{H},$ we use $\mathcal{U}(\mathcal{H})$ for the unitary operators on $\mathcal{H}.$ Given a countable, discrete group $G$ and Hilbert spaces $\mathcal{H},\mathcal{K}$ and two representations $\pi\colon G\to \mathcal{U}(\mathcal{H}),\rho\colon G\to \mathcal{U}(\mathcal{K}),$ we set
\[\Hom(\pi,\rho)=\{T\in B(\mathcal{H},\mathcal{K}):T\pi(g)=\rho(g)T\mbox{ for all $g\in G$}\}.\]

\begin{lem}\label{L:autwq} Let $G$ be a countable, discrete group and let $X$ be a compact, metrizable group with $G\cc X$ by continuous automorphisms. If $G\cc (X,m_{X})$ is ergodic, then $G\cc (X,m_{X})$ is weakly mixing.

\end{lem}

\begin{proof} For $g\in G,$ we use $\alpha_{g}$ for the action of $g$ on $X.$  If the action is not weakly mixing, then $G\actson (X\times X,m_{X}\otimes m_{X})$ is not ergodic, and so by  Lemma 1.2 of \cite{Schmidt} we can find a finite dimensional Hilbert space $\mathcal{H},$ and a nontrivial irreducible representation $\pi\colon X\times X\to \mathcal{U}(\mathcal{H}),$\
so that $G_{\pi}=\{g\in G:\pi\circ \alpha_{g}\cong \pi\}$ has finite index in $G.$ By  Proposition 2.3.23 of \cite{Kowalski}, we may write $\mathcal{H}=\mathcal{H}_{1}\otimes \mathcal{H}_{2}$ for Hilbert space $\mathcal{H}_{j},j=1,2$ and we may  find irreducible representations $\pi_{j}\colon X\to \mathcal{U}(\mathcal{H}_{j}),j=1,2$ so that
\[\pi(x,y)=\pi_{1}(x)\otimes\pi_{2}(y).\]
Since $\pi$ is not trivial, we have that one of $\pi_{1},\pi_{2}$ is not trivial. Without loss of generality, we assume that $\pi_{1}$ is not trivial. Suppose $g\in G_{\pi}$ and let $U\in \mathcal{U}(\mathcal{H}_{1}\otimes \mathcal{H}_{2})$ be such that
\[\pi(x,y)=U^{*}\pi(\alpha_{g}(x),\alpha_{g}(y))U\mbox{ for all $(x,y)\in X\times X.$}\]
For $\zeta,\zeta'\in \mathcal{H}_{2},$  define $U_{\zeta,\zeta'}\in B(\mathcal{H}_{1})$ by
\[\ip{U_{\zeta,\zeta'}(\xi),\xi'}=\ip{U(\xi\otimes \zeta),\xi'\otimes \zeta'}.\]
Since $U$ is a unitary, we may fix a choice of $\zeta,\zeta'$ so that $U_{\zeta,\zeta'}\ne 0,$ note that $U_{\zeta,\zeta'}\in \Hom(\pi_{1},\pi_{1}\circ \alpha_{g}).$ By Schur's Lemma we have
\[U_{\zeta,\zeta'}^{*}U_{\zeta,\zeta'}\in \Hom(\pi_{1},\pi_{1})=\C \id,\]
so there is a $c\in \C$ with
\[U_{\zeta,\zeta'}^{*}U_{\zeta,\zeta'}=c \id.\]
Since $U_{\zeta,\zeta'}\ne 0,$ and $U_{\zeta,\zeta'}^{*}U_{\zeta,\zeta'}$ is a positive operator, $c$ must be a positive real number. So rescaling $U_{\zeta,\zeta'}$ by $\frac{1}{\sqrt{c}}$ gives us an isometric, equivariant isomorphism $\pi_{1}\cong \pi_{1}\circ \alpha_{g}$ and this proves that $G_{\pi}\subseteq \{g\in G:\pi_{1}\circ\alpha_{g}\cong \pi_{1}\}.$  By  Lemma 1.2 of \cite{Schmidt}, it follows that $G\cc (X,m_{X})$ is not ergodic, a contradiction.

\end{proof}

\begin{cor}\label{C:corwmdq} Suppose that $G$ is a countable, discrete, sofic group with a sofic approximation $\sigma_{i}\colon G\to S_{d_{i}}.$ Let $X$ be a compact, metrizable group with $G\cc X$ by automorphisms. If $G\cc (X,m_{X})$ is ergodic and $\mu_{i}\in\Prob(X^{d_{i}})$ with $\mu_{i}\to^{le}\mu,$ then $\mu_{i}\to^{lde}\mu.$

\end{cor}

\begin{proof} This is a simple combination of  Lemma \ref{L:autwq} and Corollary \ref{C:dqfacts}.

\end{proof}

We can now list many examples of actions on compact groups which have a sequence of measures which locally and doubly empirically converge to the Haar measure.  Recall that if $X$ is a compact, metrizable group and $G$ is a countable, discrete group with $G\actson X$ by automorphisms, then the \emph{homoclinic group} of $G\actson X$ is defined by
\[\Delta(X,G)=\{x\in X:\lim_{g\to\infty}gx=1\}.\]
\begin{prop}\label{P:nonndqentropyalgexamples} Let $G$ be a countable, discrete, sofic group with sofic approximation $\sigma_{i}\colon G\to S_{d_{i}}.$  For each of the probability measure-preserving actions $G\actson (X,\mu),$ there is a sequence $\mu_{i}\in \Prob(X^{d_{i}})$ with $\mu_{i}\to^{lde}\mu:$
\begin{enumerate}[(i)]
\item  $G\actson (X,\mu)$ any Bernoulli shift action, \label{E:Bernoulli}
\item $G\actson (X,m_{X})$ where $X$ is a profinite group, the action is by continuous automorphisms, and $G\actson X$ has  a dense homoclinic group, \label{E:Profinite}
\item $G\acston (X_{f},m_{X_{f}})$ where $f\in M_{m,n}(\Z(G))$ is such that $\lambda(f)(\ell^{2}(G)^{\oplus n})$ is dense, \label{E:Me}
\item $G\actson (X_{f},m_{X_{f}})$ where $f\in M_{n}(\Z(G))$ is such that $\lambda(f)$ is injective.\label{E:Me2}
\end{enumerate}

\end{prop}

\begin{proof}
(\ref{E:Bernoulli}): This is implicitly due to Bowen in \cite{Bow} (see the proof of Theorem 8.1 therein) and was shown explicitly by Austin (see Lemma 5.11 of \cite{AustinAdd}).

(\ref{E:Profinite}):  It is known that density of the homoclinic group implies that $G\actson (X,m_{X})$ is ergodic (see e.g. Theorem 4.1 of \cite{BowenLi}). So by Corollary \ref{C:corwmdq}, it is enough to show that there exists a sequence $\mu_{i}\in \Prob(X^{d_{i}})$ with $\mu_{i}\to^{lw^{*}}m_{X},$ which is implicit in the proof of Theorem 8.2 of \cite{GabSew}.

(\ref{E:Me}):  By the arguments of \cite{LiSchmidtPet}, we know that $G\actson (X_{f},m_{X_{f}})$ is ergodic. It is implicit in the proof of Theorem 5.9 of \cite{Me5}  that there is a sequence of measures $\mu_{i}\in \Prob(X_{f}^{d_{i}})$ with $\mu_{i}\to^{le}m_{X_{f}},$ so applying Corollary \ref{C:corwmdq} completes the proof.

(\ref{E:Me2}) From the rank-nullity theorem for von Neumann dimension (see \cite{Luck} Lemma 1.13) we know that $\lambda(f)$ is injective if and only if $\lambda(f)$ has dense image for $f\in M_{n}(\Z(G)).$

\end{proof}

\section{Equality of Topological and Measure-Theoretic Entropy}\label{S:equality}

To prove the equality of topological and measure-theoretic entropy we will need a convolution argument, for which we establish the following preliminary (and well-known) fact.

\begin{lem}\label{L:convolutionconvergence}
Let $X$ be a compact, metrizable group. Given a neighborhood $\mathcal{O}$ in $\Prob(X)$ of $m_{X}$ in the weak$^{*}$-topology, there is a neighborhood $\mathcal{V}$ of $\Prob(X)$ of $m_{X}$ in the weak$^{*}$ topology so that $\mu*\nu\in\mathcal{O}$ for all $\mu\in \mathcal{V},\nu\in \Prob(X).$

\end{lem}

\begin{proof}
Since $\Prob(X)$ is compact, and $\mu*m_{X}=m_{X}$ for any $\mu\in \Prob(X),$ to prove the Lemma it is enough to show that convolution is continuous in  both variables in the weak$^{*}$-topology. Let $p\colon X\times X\to X$ be the multiplication map: $p(x,y)=xy.$ Then
\[\mu*\nu=p_{*}(\mu\otimes \nu).\]
As the maps
\[\Prob(X)\times \Prob(X)\to \Prob(X\times X),\,\,(\mu,\nu)\mapsto \mu\otimes\nu,\]
\[\Prob(X\times X)\to \Prob(X),\,\, \zeta\mapsto p_{*}\zeta\]
are continuous in the weak$^{*}$-topology, it follows that the map $(\mu,\nu)\mapsto \mu*\nu$ is continuous in the weak$^{*}$-topology.

\end{proof}

The following Lemma contains the main convolution argument we will need to establish Theorem \ref{T:mainintro}. If $X$ is a compact group and $\phi,\psi\in X^{n}$ for some $n\in \N,$ we define $\phi\psi\in X^{n}$ by $\phi\psi(j)=\phi(j)\psi(j).$

\begin{lem}\label{L:dqconv2} Let $G$ be a countable, discrete, sofic group with sofic approximation $\sigma_{i}\colon G\to S_{d_{i}}.$ Let $X$ be a compact, metrizable, group with $G\cc X$ by continuous automorphisms. Let $\mu_{i}\in \Prob(X^{d_{i}})$ with $\mu_{i}\to^{lde}m_{X}.$ Let $\rho$ be a dynamically generating pseudometric on $X.$
Then, for any finite $F\subseteq G,L\subseteq C(X),$ and $\delta,\varepsilon>0,$ there is a finite $F'\subseteq G,$ a $\delta'>0,$ and an $I\in \N$ so that for any $i>I,$ and any $\psi\in \Map(\rho,F',\delta',\sigma_{i})$ we have
\[\mu_{i}(\{\phi:\psi\phi\in \Map_{m_{X}}(\rho,F,L,\delta,\sigma_{i})\})\geq 1-\varepsilon.\]

\end{lem}

\begin{proof}

 Fix finite $F\subseteq G,L\subseteq C(X),$ and $\delta,\varepsilon>0.$ Let $F_{n}'$ be an increasing sequence of finite subsets of $G$ whose union is $G$ and let $\delta_{n}'$ be a decreasing sequence of positive real numbers converging to $0.$ If the claim is false, then we can find an increasing sequence of integers $i_{n}\to \infty$ and a $\psi_{i_{n}}\in \Map(\rho,F_{n}',\delta_{n}',\sigma_{i_{n}})$
so that
\[\mu_{i_{n}}(\{\phi:\psi_{i_{n}}\phi\in \Map(\rho,F,L,\delta,\sigma_{i_{n}})\})<1-\varepsilon.\]
Passing to a subsequence we may assume that
\[(\psi_{i_{n}})_{*}(u_{d_{i_{n}}})\to \nu\in \Prob_{G}(X).\]
Let $p\colon X\times X\to X$ be the multiplication map $p(x,y)=xy.$ Let $L'=\{f\circ p:f\in L\}.$ Let $\widetilde{\rho}$ be the metric on $X\times X$ defined by
\[\widetilde{\rho}((x_{1},x_{2}),(y_{1},y_{2}))^{2}=\frac{\rho(x_{1},y_{1})^{2}+\rho(x_{2},y_{2})^{2}}{2}.\]
By continuity of the map $p,$ we may find a $\delta'>0,$ and finite $L'\subseteq C(X),F'\subseteq G$ so that if  $\Phi \in \Map_{\nu\otimes m_{X}}(\widetilde{\rho},F,L',\delta',\sigma_{i_{n}}),$ then $p\circ \Phi\in \Map_{m_{X}}(\rho,F,L,\delta,\sigma_{i}).$  By Corollary \ref{C:dqfacts} (\ref{I:productdq}) we have 
\[\mu_{i_{n}}(\{\phi:\psi_{i_{n}}\otimes \phi\in \Map_{\nu\otimes m_{X}}(\widetilde{\rho},F',L',\delta',\sigma_{i_{n}})\})\geq 1-\varepsilon,\]
for all large $n.$
Our choice of $F',L',\delta'$ imply that
\[\mu_{i_{n}}(\{\phi:\psi_{i_{n}}\phi\in \Map_{m_{X}}(\rho,F,L,\delta,\sigma_{i_{n}})\})\geq 1-\varepsilon,\]
for all $n,$ a contradiction.

\end{proof}

We now proceed to showing that topological sofic entropy and measure-theoretic sofic entropy agree if there is a sequence $\mu_{i}\in \Prob(X^{d_{i}})$ so that $\mu_{i}\to^{lde}m_{X}.$ We recall the definition of topological entropy in the presence. This was  defined in \cite{LiLiang2} Definition 9.3 under the name ``entropy relative to an extension,'' we will use the term entropy in the presence to coincide with the terminology set out in \cite{Me9} Definition 2.7.

Let $(X,d)$ be a pseudometric space. For $A,B\subseteq X$ and $\delta>0$ we say that $A$ is $\delta$-contained in $B,$ and write $A\subseteq_{\delta}B,$ if for every $a\in A,$ there is a $b\in B$ with $d(a,b)<\delta.$ We say that $A\subseteq X$ is $\delta$-dense if $X\subseteq_{\delta}A.$ We say that $A\subseteq X$ is $\delta$-separated if for all $a\ne a'$ with $a,a'\in A$ we have $d(a,a')>\delta.$ We let $S_{\delta}(X,d)$ be the smallest cardinality of a $\delta$-dense subset of $X$ and $N_{\delta}(X,d)$ be the largest cardinality of a $\delta$-separated subset of $X.$ Note that
\[S_{\delta}(X,d)\leq N_{\delta}(X,d)\leq S_{\delta/2}(X,d).\]

\begin{defn}[Definition 9.3 in \cite{LiLiang2}] \label{D:topentrpresence} Let $G$ be a countable, discrete, sofic group with sofic approximation $\sigma_{i}\colon G\to S_{d_{i}}.$ Suppose that $X$ is a compact, metrizable space with $G\cc X$ by homeomorphisms. Let $G\cc Y$ be a topological factor with factor map $\pi\colon X\to Y$ (i.e. $G\cc Y$ by homeomorphisms and $\pi$ is a $G$-equivariant surjective map). Assume we have dynamically generating pseudometrics $\rho_{X},\rho_{Y}$  on $X,Y$ respectively. For a finite $F\subseteq G$ and $\delta>0,$ we set
\[\Map(Y:\rho_{X},F,\delta,\sigma_{i})=\{\pi\circ \phi:\phi\in \Map(\rho_{X},F,\delta,\sigma_{i})\}.\]
We define
\[h_{(\sigma_{i})_{i},\topo}(\rho_{Y}:\rho_{X},F,\delta,\varepsilon)=\limsup_{i\to\infty}\frac{1}{d_{i}}\log N_{\varepsilon}(\Map(Y:\rho_{X},F,\delta,\sigma_{i}),\rho_{Y,2}),\]
\[h_{(\sigma_{i})_{i},\topo}(\rho_{Y}:\rho_{X},\varepsilon)=\inf_{\substack{\textnormal{ finite} F\subseteq G,\\ \delta>0}}h_{(\sigma_{i})_{i},\topo}(\rho_{Y}:\rho_{X},F,\delta \varepsilon),\]
\[h_{(\sigma_{i})_{i},\topo}(Y:X,G)=\sup_{\varepsilon>0}h_{(\sigma_{i})_{i},\topo}(\rho_{Y}:\rho_{X},\varepsilon).\]
We call $h_{(\sigma_{i})_{i},\topo}(Y:X,G)$ the topological entropy of $Y$ in the presence of $X.$
\end{defn}

Note that if $Y=X$ with factor map $\id,$ then this is just the topological entropy. We also recall the formulation of measure-theoretic entropy in the presence we formulated in \cite{Me9}.

\begin{defn}[Definition 2.7 in \cite{Me9}]\label{D:measentpresence} Let $G$ be a countable, discrete, sofic group with sofic approximation $\sigma_{i}\colon G\to S_{d_{i}}.$ Let $X$ be a compact, metrizable space with $G\cc X$ by homeomorphisms and fix a $\mu\in \Prob_{G}(X).$ Let $G\cc Y$ be a topological factor with factor map $\pi\colon X\to Y$ (i.e. $G\cc Y$ by homeomorphisms and $\pi$ is a $G$-equivariant surjective map) and set $\nu=\pi_{*}\mu.$ Let $\rho_{X},\rho_{Y}$ be dynamically generating pseudometrics on $X,Y$ respectively. For finite $F\subseteq G,L\subseteq C(X),$ and $\delta>0$ we set
\[\Map_{\mu}(Y:\rho_{X},F,L,\delta,\sigma_{i})=\{\pi\circ \phi:\phi\in \Map_{\mu}(\rho_{X},F,L,\delta,\sigma_{i})\}.\]
We define
\[h_{(\sigma_{i})_{i},\mu}(\rho_{Y}:\rho_{X},F,L,\delta,\varepsilon)=\limsup_{i\to\infty}\frac{1}{d_{i}}\log N_{\varepsilon}(\Map_{\mu}(Y:\rho_{X},F,L,\delta,\sigma_{i}),\rho_{Y,2}),\]
\[h_{(\sigma_{i})_{i},\mu}(\rho_{Y}:\rho_{X},\varepsilon)=\inf_{\substack{\textnormal{ finite} F\subseteq G,\\ \textnormal{finite} L\subseteq C(X),\\  \delta>0}}h_{(\sigma_{i})_{i},\mu}(\rho_{Y}:\rho_{X},F,L,\delta,\varepsilon),\]
\[h_{(\sigma_{i})_{i},\mu}(Y:X,G)=\sup_{\varepsilon>0}h_{(\sigma_{i})_{i},\mu}(\rho_{Y}:\rho_{X},\varepsilon).\]
We call $h_{(\sigma_{i})_{i},\mu}(Y:X,G)$ the measure-theoretic entropy of $Y$ in the presence of $X.$
\end{defn}
Again if $Y=X$ with factor map $\id,$ this is just the measure-theoretic entropy. Recall that if $X$ is a compact group, then there is a  metric $d$ on $X$ giving the topology and such that $d(axb,ayb)=d(x,y)$ for all $a,b,x,y\in X$, we call $d$ a \emph{bi-invariant metric}.
\begin{thm}\label{T:maximize} Let $G$ be a countable, discrete group with sofic approximation $\sigma_{i}\colon G\to S_{d_{i}}.$ Let $X$ be a compact, metrizable group with $G\cc X$ by continuous automorphisms. Suppose that there exists $\mu_{i}\in\Prob(X^{d_{i}})$ with $\mu_{i}\to^{lde}m_{X}.$
Then
\[h_{(\sigma_{i}),\topo}(X/Y:X,G)=h_{(\sigma_{i})_{i},m_{X}}(X/Y:X,G),\]
for any closed, normal, $G$-invariant subgroup $Y$ of $X.$

\end{thm}

\begin{proof}

   By assumption, we may find a $\mu_{i}\in \Prob(X^{d_{i}})$ with $\mu_{i}\to^{lde}m_{X}.$  Fix compatible bi-invariant metrics $\rho_{X},\rho_{X/Y}$ on $X,X/Y,$  and let $\pi\colon X\to X/Y$ be the quotient map. It is clear that
\[h_{(\sigma_{i}),\topo}(X/Y:X,G)\geq h_{(\sigma_{i})_{i},m_{X}}(X/Y:X,G).\]
For the reverse inequality, fix $\varepsilon,\delta>0,$ and finite $F\subseteq G,L\subseteq C(X).$  By Lemma \ref{L:dqconv2} we may find a $\delta'>0,$ and a finite $F'\subseteq G$ so that for all sufficiently large $i,$ and for all $\psi\in \Map(\rho_{X},F',\delta',\sigma_{i}),$ 
\[\mu_{i}(\{\phi:\psi \phi\in \Map_{m_{X}}(\rho_{X},F,L,\delta,\sigma_{i})\})\geq 1-\varepsilon.\]
Choose a $S\subseteq \Map(\rho,F',\delta',\sigma_{i})$ so that $\rho_{2,X/Y}(\pi\circ \psi,\pi\circ \psi')\geq \varepsilon$ for $\psi,\psi'\in S$ with $\psi\ne \psi',$  and such that
\[|S|=N_{\varepsilon}(\Map(X/Y:\rho_{X},F',\delta',\sigma_{i}),\rho_{2,X/Y}).\]
By Fubini's theorem, we have
\[u_{S}\otimes \mu_{i}(\{(\psi,\phi):\psi \phi\in \Map_{m_{X}}(\rho_{X},F,L,\delta,\sigma_{i})\})\geq 1-\varepsilon,\]
for all sufficiently large $i.$
By Fubini's theorem again, for all sufficiently large $i$ we may find a $\phi \in X^{d_{i}}$ so that
\[u_{S}(\{\psi:\psi\phi\in \Map_{m_{X}}(\rho_{X},F,L,\delta,\sigma_{i})\})\geq 1-\varepsilon.\]
For such $i,\phi$ set
\[S_{0}=\{\psi:\psi \phi\in \Map_{m_{X}}(\rho_{X},F,L,\delta,\sigma_{i})\}.\]
For $\psi,\psi'\in S_{0}$ with $\psi\ne\psi',$ translation-invariance of $\rho_{X/Y}$ implies that
\[\rho_{2,X/Y}((\pi\circ \psi)(\pi\circ \phi),(\pi\circ \psi')(\pi\circ\phi))=\rho_{2,X/Y}(\pi\circ \psi,\pi\circ \psi')\geq \varepsilon.\]
Hence,
\[N_{\varepsilon}(\Map_{m_{X}}(X/Y:\rho_{X},F,L,\delta,\sigma_{i}),\rho_{2,X/Y})\geq |S_{0}|\geq (1-\varepsilon)|S|=(1-\varepsilon)N_{\varepsilon}(\Map(X/Y:\rho_{X},F',\delta',\sigma_{i}),\rho_{2,X/Y}).\]
Since the above inequality holds for all sufficienty large $i,$ it follows that
\[h_{(\sigma_{i})_{i},m_{X}}(\rho_{X/Y}:\rho_{X},F,L,\delta,\varepsilon)\geq h_{(\sigma_{i})_{i}}(\rho_{X/Y}:\rho_{X},F',\delta',\varepsilon)\geq h_{(\sigma_{i})_{i}}(\rho_{X/Y}:\rho_{X},\varepsilon).\]
Infimizing over $F,L,\delta$ proves that
\[h_{(\sigma_{i})_{i},m_{X}}(\rho_{X/Y}:\rho_{X},\varepsilon)\geq h_{(\sigma_{i})_{i}}(\rho_{X/Y}:\rho_{X},\varepsilon),\]
and taking the supremum over all $\varepsilon>0$ completes the proof.

\end{proof}
A special case is when $Y=\{1\},$ which gives us the following (compare \cite{GabSew} Theorem 8.2).
\begin{cor}Let $G$ be a countable, discrete group with sofic approximation $\sigma_{i}\colon G\to S_{d_{i}}.$ Let $X$ be a compact, metrizable group with $G\cc X$ by continuous automorphisms. If there exists a sequence $\mu_{i}\in\Prob(X^{d_{i}})$ with $\mu_{i}\to^{lde}m_{X},$
 then
\[h_{(\sigma_{i}),\topo}(X,G)=h_{(\sigma_{i})_{i},m_{X}}(X,G).\]\end{cor}

In this specific case we can say even more, due to the following Lemma.

\begin{lem}\label{L:dqconv}

 Let $G$ be a countable, discrete, sofic group with sofic approximation $\sigma_{i}\colon G\to S_{d_{i}}.$ Let $X$ be a compact, metrizable, group with $G\cc X$ by automorphisms. Suppose that $\mu_{i}\in \Prob(X^{d_{i}})$ with $\mu_{i}\to^{le}m_{X}.$ If $\nu_{i}\in \Prob(X^{d_{i}})$ is a sequence so that for all $F\subseteq G$ finite and $\delta>0$ we have
\[\nu_{i}(\Map(\rho,F,\delta,\sigma_{i}))\to 1,\]
then $\nu_{i}*\mu_{i}\to^{lde}m_{X}.$
\end{lem}

\begin{proof}

Let $\rho$ be a  continuous, translation-invariant metric on $X.$
We first show that $\nu_{i}*\mu_{i}\to^{le}m_{X}.$ We do this in two parts.

\emph{Part 1: We show that given finite $F\subseteq G, L\subseteq C(X),$ and a $\delta>0$ we have}
\[\nu_{i}*\mu_{i}(\Map_{m_{X}}(\rho,F,L,\delta,\sigma_{i}))\to_{i\to\infty}1.\]
To prove this fix $\varepsilon>0.$ Let $F',\delta'$ be as in Lemma \ref{L:dqconv2} for this $F,L,\delta,\mu_{i}.$ Lemma \ref{L:dqconv2} and Fubini's theorem then imply that 
\[\nu_{i}*\mu_{i}(\Map_{m_{X}}(\rho,F,L,\delta,\sigma_{i}))\geq (1-\varepsilon)\mu_{i}(\Map(\rho,F,L,\delta,\sigma_{i}))\]
for all sufficiently large $i.$ Thus
\[\liminf_{i\to\infty}\nu_{i}*\mu_{i}(\Map_{m_{X}}(\rho,F,L,\delta,\sigma_{i}))\geq 1-\varepsilon.\]
And since $\varepsilon>0$ is arbitrary, this proves Part 1.

\emph{Part 2: We show that $\nu_{i}*\mu_{i}\to^{lw^{*}}m_{X}.$}
Since $(\mathcal{E}_{j})_{*}(\mu_{i}*\nu_{i})=(\mathcal{E}_{j})_{*}(\mu_{i})*(\mathcal{E}_{j})_{*}(\nu_{i}),$ this is an easy exercise using Lemma \ref{L:convolutionconvergence} and the fact that $\mu_{i}\to^{lw^{*}}m_{X}.$

To show that $\mu_{i}*\nu_{i}\to^{lde}m_{X},$  let $\widetilde{\rho}$ be the continuous, translation-invariant metric on $X\times X$ given by
\[\widetilde{\rho}((x_{1},y_{1}),(x_{2},y_{2}))=\left(\frac{1}{2}(\rho(x_{1},x_{2})^{2}+\rho(y_{1},y_{2})^{2})\right)^{1/2}.\]
It is straightforward to show that for any finite $F\subseteq G$ and $\delta>0,$ we have that $\nu_{i}\otimes \nu_{i}(\Map(\widetilde{\rho},F,\delta,\sigma_{i}))\to 1.$ Since $\mu_{i}\otimes \mu_{i}\to^{le}m_{X}\otimes m_{X}=m_{X\times X}$ and 
\[(\nu_{i}*\mu_{i})\otimes (\nu_{i}*\mu_{i})=(\nu_{i}\otimes \nu_{i})*(\mu_{i}\otimes \mu_{i})\]
it follows from the first half of the proof that $(\nu_{i}*\mu_{i})*(\nu_{i}\otimes \mu_{i})\to^{le}m_{X}\otimes m_{X},$ i.e. that $\nu_{i}*\mu_{i}\to^{lde}m_{X}.$ 

\end{proof}

\begin{thm}\label{T:ergdqeqldahglahdg} Let $G$ be a countable, discrete, sofic group with sofic approximation $\sigma_{i}\colon G\to S_{d_{i}}.$ Let $X$ be a compact, metrizable, group with $G\cc X$ by automorphisms. If there exists $\mu_{i}\in \Prob(X^{d_{i}})$ and $\mu_{i}\to^{le}m_{X},$ then
\[h_{(\sigma_{i})_{i},m_{X}}^{lde}(X,G)=h_{(\sigma_{i})_{i},m_{X}}(X,G)=h_{(\sigma_{i})_{i},\topo}(X,G).\]
\end{thm}

\begin{proof} Let $\rho$ be a bi-invariant metric on $X.$ It is enough to show that  $h_{(\sigma_{i})_{i},\topo}(X,G)\leq h_{(\sigma_{i})_{i},m_{X}}^{lde}(X,G).$
Fix an increasing sequence of finite subsets $F_{n}$ of $G$ whose union is $G,$ a decreasing sequence $\delta_{n}$ of positive real numbers tending to zero, and an $\varepsilon>0.$ We may find a strictly increasing sequence $i_{n}$ of natural numbers with
\[\frac{1}{d_{i_{n}}}\log N_{\varepsilon/2}(\Map(\rho,F_{n},\delta_{n},\sigma_{i_{n}}),\rho_{2})\geq -2^{-n}+h_{(\sigma_{i})_{i}}(\rho,\varepsilon/2).\]
Let $S_{i_{n}}$ be an $\frac{\varepsilon}{2}$-separated subset of $\Map(\rho,F_{n},\delta_{n},\sigma_{i_{n}})$ with respect to $\rho_{2}$  of maximal cardinality.  By Lemma \ref{L:dqconv} we have that
\[u_{S_{i_{n}}}*\mu_{i_{n}}\to^{lde}m_{X}.\]
Fix $\varepsilon,\delta>0,$ and an integer $n,$ and suppose that $A\subseteq X^{d_{i_{n}}}$ has $u_{S_{i_{n}}}*\mu_{i_{n}}(A)\geq 1-\delta.$ Note that
\[u_{S_{i_{n}}}*\mu_{i_{n}}(A)=u_{S_{i_{n}}}\otimes \mu_{i_{n}}(\{(\phi,\psi):\phi \psi\in A\},\]
so we may find a $\psi\in X^{d_{i}}$ so that
\[u_{S_{i_{n}}}(\{\phi:\phi\psi \in A\})\geq 1-\delta.\]
Let $S_{i_{n}}^{0}=\{\phi:\phi\psi \in A\},$ by bi-invariance of $\rho$ we have that $S_{i_{n}}^{0}\psi$ is $\frac{\varepsilon}{2}$-separated and contained in $A,$ so
\[S_{\varepsilon/4}(A,\rho_{2})\geq N_{\varepsilon/2}(A,\rho_{2})\geq |S_{i_{n}}^{0}|\geq (1-\delta)N_{\varepsilon/2}(\Map(\rho,F_{n},\delta_{n},\sigma_{i_{n}}),\rho_{2}).\]
Since $S_{\varepsilon/4,\delta}(u_{S_{i_{n}}}*\mu_{i_{n}},\rho_{2})$ is the minimum of $S_{\varepsilon/4}(A,\rho_{2})$ over all sets $A\subseteq X^{d_{i_{n}}}$ with $u_{S_{i_{n}}}*\mu_{i_{n}}(A)\geq 1-\delta,$ we have
\[S_{\varepsilon/4,\delta}(u_{S_{i_{n}}}*\mu_{i_{n}},\rho_{2})\geq  (1-\delta)N_{\varepsilon/2}(\Map(\rho,F_{n},\delta_{n},\sigma_{i_{n}}),\rho_{2})\geq \exp(d_{i_{n}}(-2^{-n})+d_{i_{n}}h_{(\sigma_{i})_{i}}(\rho,\varepsilon/2)).\]
We thus see that
\[\limsup_{n\to\infty}\frac{1}{d_{i_{n}}}\log S_{\varepsilon/4,\delta}(u_{S_{i_{n}}}*\mu_{i_{n}},\rho_{2})\geq h_{(\sigma_{i})_{i}}(\rho,\varepsilon/2).\]
Taking the supremum over $\varepsilon,\delta>0$ completes the proof.

\end{proof}

We now close with a concrete list of actions for which measure-theoretic and topological entropy (as defined by Bowen, Kerr-Li and Austin) agree.

\begin{cor} Let $G$ be a countable, discrete, sofic group with sofic approximation $\sigma_{i}\colon G\to S_{d_{i}}.$ Let $X$ be a compact, metrizable, group with $G\cc X$ by automorphisms. Suppose that one of the following cases hold:
\begin{enumerate}[(i)]
\item $X=X_{f}$ for some $f\in M_{m,n}(\Z(G))$ with $\lambda(f)$ having dense image,
\item $X=X_{f}$ for some $f\in M_{n}(\Z(G))$ with $\lambda(f)$ injective,
\item $X$ is a profinite group and $G\cc X$ has a dense homoclinic group,
\end{enumerate}
then
\[h_{(\sigma_{i})_{i},m_{X}}^{lde}(X,G)=h_{(\sigma_{i})_{i},m_{X}}(X,G)=h_{(\sigma_{i})_{i},\topo}(X,G).\]

\end{cor}

\begin{proof} We saw in Proposition \ref{P:nonndqentropyalgexamples} that each of these actions has a sequence $\mu_{i}\in \Prob(X^{d_{i}})$ with $\mu_{i}\to^{lde}m_{X},$ so this follows from Theorem \ref{T:ergdqeqldahglahdg}.

\end{proof}

\end{document}